\documentclass[11pt]{amsart}
\usepackage{amsmath, amssymb, amscd, amsthm, amsfonts, latexsym}
\usepackage{graphicx}
\usepackage{tikz-cd}
\usepackage{hyperref}
\usepackage{mathtools}
\usepackage{mathrsfs}
\usepackage{enumerate}
\usepackage{mathabx}
\usepackage{setspace}
\usepackage{color}
\usepackage{comment}
\usepackage{tikz}
\usepackage{tikz-cd}

\theoremstyle{plain}
\newtheorem{thm}{Theorem}[section]

\newtheorem{lemm}[thm]{Lemma}
\newtheorem{cor}[thm]{Corollary}

\theoremstyle{definition}
\newtheorem{defn}[thm]{Definition}
\newtheorem{exmp}[thm]{Example}

\newtheorem{ques}[thm]{Question}

\def\Z{{\mathbb Z}}

\def\cd{\protect\operatorname{cd}}
\def\cat{\protect\operatorname{cat}}

\title{On the LS-category of homomorphism of groups with torsion}
\author{Nursultan Kuanyshov}

\address{Nursultan Kuanyshov, Department of Mathematics, University of Florida, 358 Little Hall, Gainesville, FL 32611-8105, USA; Institute of Mathematics and Mathematical Modeling
125 Pushkin str., 050010 Almaty, Kazakhstan}
\email{nkuanyshov@ufl.edu, kuanyshov.nursultan@gmail.com}

\subjclass[2020]{Primary 55M30; Secondary 20J06, 20K30, 20K10, 55M25, 57R65, 57R67}

\keywords{Lusternik-Schnirelmann category, Cohomological dimension, group homomorphism}

\begin{document}
\maketitle
\begin{abstract}
We prove the equality $\cat(\phi)=\cd(\phi)$ for homomorphisms $\phi:\Gamma\to \Lambda$ between finitely generated abelian groups $\Gamma$ and $\Lambda$, where $\phi(T(\Gamma))=0$ for the torsion subgroups $T(\Gamma)$ of $\Gamma$.   
\end{abstract}

\section{Introduction}

We recall that the (reduced) Lusternik–Schnirelmann category $\cat(\phi)$ of a group homomorphism
$\phi:\Gamma\to\Lambda$ is defined as the minimum number $ k$ for which there exists an open cover
of $B\Gamma$ by $k +1$ subsets $U_0, \dots ,U_k$ such that the restriction $\bar\phi|_{U_i}$ is nullhomotopic
for all $i$ where $\bar\phi:B\Gamma\to B\Lambda$ is a map that induces $\phi$ on the fundamental group \cite{LS, CLOT}.

The cohomological dimension $\cd(\phi)$ of the group homomorphism $\phi$ is defined as
the supremum of $k$ for which there exists a $\Lambda$-module $M$ such that the induced
map on cohomology
$\phi^* : H^k(\Lambda,M) \to  H^k(\Gamma, M)$
is non-trivial \cite{Gr, Br}.

The two invariants are related by the inequality $\cd(\phi)\leq \cat(\phi)$ (see \cite{DK}). Note that for the identity homomorphism on $\Gamma$, we recover the classical invariants $\cat(\Gamma)$ and $\cd(\Gamma)$ of the group $\Gamma$. In this sense, the definitions above are generalisations to the relative setting of group homomorphisms. The paper addresses relative versions of two classical results:

First, it is a famous result of Eilenberg–Ganea \cite{EG} from the 1950’s that the equality $\cat(\Gamma)=\cd(\Gamma)$ holds for all groups $\Gamma$. The following natural question was asked by Jamie Scott \cite{Sc} in 2020.

\begin{ques}\label{question}
For which group homomorphisms $\phi:\Gamma\to\Lambda$ does the equality $$\cat(\phi)=\cd(\phi) $$
hold?
\end{ques}

It is known that not all homomorphisms satisfy Question \ref{question} (see \cite{Gr,DK,DD}). On the other hand, Jamie Scott ~\cite{Sc} proved Question \ref{question} holds for monomorphisms of any groups and for homomorphisms of free and free abelian groups. 
In the joint paper with Dranishnikov~\cite{DK}, we reduced the Question \ref{question} from arbitrary homomorphisms to epimorphisms and proved the Question \ref{question} for epimorphisms between finitely generated, torsion-free nilpotent groups. Recently, the author proved the Question \ref{question} holds for epimorphisms of almost nilpotent groups \cite{Ku1} and for epimorphisms of surface groups, which are the fundamental group of orientable, closed surfaces \cite{Ku2}.

Let $T(\Gamma)$ denote the torsion subgroups  of an abelian group $\Gamma$.
The main result of our paper is the following.
\begin{thm}\label{main}
For any homomorphism $\phi:\Gamma\to\Lambda$  
between finitely generated abelian groups satisfying the condition $\phi(T(\Gamma))=0$,  $$\cat(\phi)=\cd(\phi).$$
\end{thm}

The condition $\phi(T(\Gamma))=0$ cannot be dropped in view of the following
\begin{exmp}[A. Dranishnikov]\label{Example}
For an  epimorphism $\phi:\mathbb Z_{p^2}\to\mathbb Z_p$,
$$
\cd(\phi)=2\ \ \ \text{and}\ \ \ \cat(\phi)=\infty.
$$
\end{exmp}
\begin{proof}
The proof of the inequality $\cd(\phi\le 2$ is based on the following unpublished observation by Bestvina and Edwards: {\em There is a degree 0 map $$f:L^3_{p^2}\to L^3_p$$ of lens spaces  that induces an epimorphism $f_*:\mathbb Z_{p^2}\to\mathbb Z_p$ of the fundamental groups.} 

Here is the construction of $f$. 
We consider two circles $S^1$ with $\mathbb Z_p$ and $\mathbb Z_{p^2}$ free actions.
Note that a $\mathbb Z_{p^2}$-equivariant map $\psi:S^1\to S^1$ between them has degree $p$. Then the join product $\psi\ast\psi:S^1\ast S^1\to S^1\ast S^1$ has degree $p^2$.
Hence it induces a map of the orbit spaces $\bar\psi:L^3_{p^2}\to L^3_p$ of degree $p$. Let $q_k:S^3\to L_{p^k}^3$ denote the projection onto the orbit space.

We define $f$ as the composition
$$
L^3_{p^2}=L^3_{p^2}\#S^3\to L^3_{p^2}\vee S^3\stackrel{1\vee r}\to L^3_{p^2}\vee S^3\stackrel{\bar\psi\vee q_1}\to L^3_{p}\vee L^3_{p}\stackrel{j}\to L^3_p
$$
where  $r:S^3\to S^3$ has degree -1 and $j$ identifies two copies of $L^3_p$.
Thus, $deg(f)=p-p=0$.

Claim 1. The induced homomorphism $f^*:H^3(L^3_p;M)\to H^3(L^3_{p^2};M)$ is trivial for any $\mathbb Z_p$-module $M$.

Indeed, if $f^*(\alpha)\ne 0$ for $\alpha\in H^3(L_p^3;M)$ then by the Poincare duality with local coefficients $[L^3_{p^2}]\cap f^*(\alpha)\ne 0$. Since $f$ induces an epimorphism of the fundamental groups, the induced homomorphism for 0-homology $f_*:H_0(L^3_{p^2};M) \to H_0(L^3_p;M)$ is an isomorphism.We obtain a contradiction 
$$0\ne f_*([L^3_{p^2}]\cap f^*(\alpha))=f_*([L^3_{p^2}])\cap\alpha= 0.$$

The map $f$ extends to a map $\phi:L^\infty_{p^2}\to L^\infty_{p}$.

Claim 2. The induced homomorphism $\phi^*:H^3(L^\infty_p;M)\to H^3(L^\infty_{p^2};M)$ is trivial for any $\mathbb Z_p$-module $M$.

This follows from the fact that the inclusion homomorphism $H^k(X;M)\to H^k((X^{(k)};M)$ is injective for any CW-complex $X$ and a $\pi_1(X)$-module $M$.

The shift of dimension for group cohomology~\cite{Br} implies that $$\phi^*:H^k(L^\infty_p;M)\to H^k(L^\infty_{p^2};M)$$ is trivial for all $k\ge 3$.
Thus, $\cd(\phi)\le 2$. Using integral cohomology one can check that $\cd(\phi)\ge 2$.

Next we show that $\cat(\phi)=\infty$. This follows from the fact that the reduced K-theory cup-length of $\phi$ is unbounded. By the Atiyah's computations~\cite{AS}, $$K(B\mathbb Z_{p^k})=\mathbb Z[[\eta_k]]/(\eta_k^{p^k}-1)$$ where $\eta_k$ is the pull-back of the canonical line bundle $\eta$ over $\mathbb CP^\infty$ under the inclusion
$\mathbb Z_{p^k}\to S^1$ and $\mathbb Z[[x]]$ is the ring of formal series. The reduced K-theory of $B\mathbb Z_{p^k}$ is the subring generated by $\eta_k-1$.
The induced homomorphism $$\phi^*:K(B\mathbb Z_p)\to K(B\mathbb Z_{p^2})$$  takes $\eta_1$ to $\eta_2^p$. Then it takes $(\eta_1-1)^m$ to 
$(\eta_2^p-1)^m$. To see that $(\eta_2^p-1)^m\ne 0$ in $\mathbb Z[[\eta_2]]/(\eta_2^{p^2}-1)$ for all $m$ it suffices to show that it is not zero in
$\mathbb Z_p[[\eta_2]]/(\eta_2^{p^2}-1)$ for $m=pk+1$. Note that $$(\eta_2^p-1)^{pk+1}=\eta_2^p-1\ne 0$$ in $\mathbb Z_p[[\eta_2]]/(\eta_2^{p^2}-1)$.
This implies $\cat(\phi)\ge m$ for any $m$ (see Proposition 5.1 in~\cite{D}).
\end{proof}

\section{Preliminaries}

In this section we recall some classic theorems used in the paper.

Given positive integers $m,n \geq 1,$ we denote by $M_{m\times n}(\Z)$the set of $m\times n$ matrices with integer entries. 

\begin{thm}[The Smith normal form]\label{SNF} Given a nonzero matrix $A \in M_{m\times n}(\Z)$, there exist invertible matrices $P \in M_{m \times m}(\Z)$ and  $Q \in M_{n\times n}(\Z)$ such that 

\[PAQ=
  \left[ {\begin{array}{ccccccc}
    n_{1}  & 0  & \cdots & 0            & 0 & \cdots &0\\
    0 & n_{2}  & \cdots & 0             & 0 & \cdots &0\\
    \vdots & \vdots  & \ddots & \vdots  & \vdots & \ddots&\vdots\\
    0 & 0  & \cdots & n_{k}             & 0 & \cdots &0\\
     0 & 0  & \cdots & 0             & 0 & \cdots &0\\
    \vdots & \vdots  & \ddots & \vdots  & \vdots & \ddots&\vdots\\
    0 & 0  & \cdots & 0                 & 0 & \cdots&0\\
  \end{array} } \right]
\]
where the integer $n_{i}\geq 1$ are unique up to sign and satisfy $n_{1}|n_{2}|\cdots |n_{k}.$ Further, one can compute the integers $n_{i}$ by the recursive formula $n_{i}=\dfrac{d_{i}}{d_{i-1}},$ where $d_{i}$ is the greatest common divisor of all $i\times i$-minors of the matrix A and $d_{0}$ is defined to be 1.
\end{thm}

The proof of Theorem~\ref{SNF} can be found in ~\cite[Proposition 2.11, p.339]{Hu}. 

\begin{cor}{\label{square}} Given a matrix  $A\in M_{n\times n}(\Z)$ with $det(A)\neq 0,$  there exist invertible matrices  $P\in M_{n\times n}(\Z)$ and $Q\in M_{n\times n}(\Z)$ such that 

\[PAQ=
  \left[ {\begin{array}{ccccccc}
    1  & 0  & \cdots & 0            & 0 & \cdots &0\\
    0 & 1  & \cdots & 0             & 0 & \cdots &0\\
    \vdots & \vdots  & \ddots & \vdots  & \vdots & \ddots&\vdots\\
    0 & 0  & \cdots & 1             & 0 & \cdots &0\\
     0 & 0  & \cdots & 0             & n_{1} & \cdots &0\\
    \vdots & \vdots  & \ddots & \vdots  & \vdots & \ddots&\vdots\\
    0 & 0  & \cdots & 0                 & 0 & \cdots& n_{k}\\
  \end{array} } \right]
\]
where the integer $n_{i}\geq 2$ are unique up to sign and satisfy $n_{1}|n_{2}|\cdots |n_{k}.$
\end{cor}

\begin{thm}[Invariant Factor Decomposition (IDF) for Finite Abelian Groups]{\label{IDF}} Every finite abelian group $\Gamma$ can be written uniquely as $\Gamma=\Z_{n_{1}}\times...\times \Z_{n_{k}}$ where the integers $n_{i}\geq 2$ are the invariant factors of $\Gamma$ that satisfy $n_{1}|n_{2}|...|n_{k}$ and $\Z_{n_{i}}$ are cyclic group of order $n_{i}, i=1,\cdots,k.$
\end{thm}

The proof of Theorem~\ref{IDF} can be found in ~\cite[Theorem 3, p.158]{DF}. Alternatively, one can apply Corollary ~\ref{square} and get the result. 

\begin{defn}
Given a finite abelian group $\Gamma$, the Smith normal number $k(\Gamma)$ of $\Gamma$ is the number $k$ from Theorem ~\ref{IDF}.     
\end{defn}

In the proof of our main result about homomorphism between finite groups we use Shapiro's Lemma ~\cite[Proposition 6.2, page 73]{Br}. 

\begin{thm}[Shapiro's Lemma]{\label{Shapiro}}
    If $i:\Gamma\to \Lambda$ is a monomorphism and $M$ is an $\Gamma$-module, then the composition
     $$H^{*}(\Lambda,Coind_{\Gamma}^{\Lambda}M)\stackrel{i^*}\to H^*(\Gamma,Coind_{\Gamma}^{\Lambda}M)\stackrel{\alpha_*}\to H^*(\Gamma,M)$$
     is an isomorphism,
where $Coind_{\Gamma}^{\Lambda}M=Hom_{\Z \Gamma}(\Z \Lambda,M)$ and the homomorphism of coefficients $\alpha:Hom_{\Z \Gamma}(\Z \Lambda,M)\to M$ is defined as $\alpha(f)=f(e)$.
\end{thm}

 In the paper we use the notation $H^*(\Gamma,A)$ for the cohomology of a group $\Gamma$ with coefficients in a $\Gamma$-module $A$. The cohomology groups with twisted coefficients in $A$ of a space $X$ with fundamental group $\Gamma$ we denote as $H^*(X;A)$. Thus, $H^*(\Gamma,A)=H^*(B\Gamma;A)$ where $B\Gamma=K(\Gamma,1)$. 
 
 We say that a CW-complex $X$ is of {\em finite type} if each $n$-skeleton $X^{(n)}$ of   $X$ is finite.

The Kunneth Formula for cohomology of product of two spaces with field coefficient $F$  can be found in Spanier ~\cite[Theorem 5.5.11]{Sp} or Dold ~\cite[Proposition VI.12.16]{Do}. We state a special case, which will be used in the paper. 

\begin{thm}[The Kunneth Formula]\label{KFT} Let $F$ be a field. If the CW-complexes X, Y are of finite type, then the cross product 
$$\underset{k}{\bigoplus}H^{k}(X;F)\otimes_{F} H^{n-k}(Y;F)\overset{\times}{\to} H^{n}(X\times Y;F)$$
is an isomorphism.     
\end{thm}

We recall the Universal Coefficient Formula (UCF) for cohomology (see~\cite[Theorem 5.5.10]{Sp}).

{\bf{The Universal Coefficient Formula}}: {\em Let the CW-complex X be of finite type and a coefficient group $\Gamma,$ then for each n there is the short exact sequence

$$0\to H^{n}(X){\otimes}_{\Z} \Gamma\to H^{n}(X;\Gamma) \to Tor(H^{n+1}(X),\Gamma)\to 0,$$ which splits.}

\section{Finitely Generated Abelian Groups}

\begin{lemm}{\label{Main}}
Given an epimorphism $\phi:\Z^n\to \Lambda$ with a finite group $G$ there is an epimorphism $\pi:\Z^n\to\Z^k$ such that $\phi=\psi\circ\pi$ where $k=k(\Lambda)$ is the Smith normal number for $\Lambda$, $$\psi=\prod_{i=1}^{k}(\psi_i:\Z\to\Z_{n_i})   ,$$  and the numbers $n_1|\dots|n_k$ are taken from IFD for $\Lambda$ from Theorem ~\ref{IDF}.

\end{lemm}

\begin{proof}
Being a subgroup of $\Z^n$,
the kernel $\ker\phi$ is a free abelian group. Since $G$ is finite,
$\ker\phi$ is isomorphic to $\Z^n$. We fix a basis in $\ker\phi$.
Let $A:\Z^n\to \ker\phi$ be an isomorphism. We regard $A:\Z^n\to\Z^n$ as the embedding.
Then $A$ is given by $n\times n$ matrix the columns of which form our basis. We apply Corollary ~\ref{square} (Smith Normal Form) to get matrices $Q$ and $P$ that change in a special way the bases in the domain of $A$ and the range of $\phi$ respectively. Thus, $AQ(\Z^n)=A(\Z^n)=\ker\phi$.
Then $PAQ(\Z^n)=\ker(\phi P^{-1})$. Then $$G\cong (\phi P^{-1})(\Z^n)=\Z^n\Big/\ker(\phi P^{-1})=\Z^n\Big/PAQ(\Z^n)=$$
$$=\left(\Z^{n-k}\Big/\langle 1\rangle\times\cdots\times\langle 1\rangle\right)\times\left(\Z^k\Big/\langle n_{1}\rangle\times\langle n_{2}\rangle\times \cdots\times\langle n_{k}\rangle\right)=$$
$$=\left(\Z\Big/\Z\times\cdots\Z\Big/\Z\right)\times\left(\Z\Big/n_{1}\Z\times\cdots\times\Z\Big/n_{k}\Z\right)=$$
$$=pr_k(\Z^n)\Big/\langle n_{1}\rangle\times \cdots\times\langle n_{k}\rangle=\psi pr_k(\Z^n)$$ where $pr_k:\Z^n\to\Z^k$ is the projection onto the last k coordinates. Thus, $\phi P^{-1}=\psi pr_k$. Then $\phi=\psi \pi$ with $\pi=pr_k P$.

\end{proof}

\begin{lemm}{\label{Smith}}
Let $\phi:\Z^{n}\to \Lambda$ be an epimorphism, where $\Lambda$ is a finite abelian group. Then $\cat(\phi)=\cd(\phi).$ In particular, $\cat(\phi)=\cd(\phi)=k(\Lambda)$ where k is the Smith normal number for given a finite abelian group $\Lambda$.
\end{lemm} 

\begin{proof}
Since $\cd(\phi)\leq\cat(\phi)$ for any group homomorphism ~\cite{DK}, we just need to show two inequalities, i.e $\cat(\phi)\leq k(\Lambda)$ and $k(\Lambda)\leq \cd(\phi).$ Then, observing the chain inequalities $k(\Lambda)\leq\cd(\phi)\leq\cat(\phi)\leq k(\Lambda),$ we obtain the conclusion of Lemma \ref{Smith}. 

 Let us show the first inequality $\cat(\phi)\leq k(\Lambda).$ Let $k=k(\Lambda).$ By Lemma ~\ref{Main}, there exists an epimorphisms $\pi:\Z^{n}\to\Z^{k}$ and $\psi:\Z^{k}\to \Lambda$ such that we have the following commutative diagram:  
$$
\begin{tikzcd}
 &  \Z^{n} \arrow[dl, "\pi"] \arrow[d, "\phi"]&  \\
\Z^{k} \arrow[r, "\psi"] & \Lambda \arrow[r] & 0 . \\
\end{tikzcd} 
$$
Using well-known facts on the LS-category $\cat$ ~\cite{CLOT}, we obtain:
$$\cat(\phi)\leq\min\{\cat(\psi),\cat(\pi))\}\leq \cat(\psi)\leq \min\{\cat(T^{k}),\cat(B\Lambda)\}\leq\cat(T^{k})=k$$ where $T^{k}$ is the $k$ dimensional torus. 

Since $B\pi:T^{n}\to T^{k}$ is a retraction, $\pi$ is injective on cohomology, so we have $\cd(\phi)=\cd(\psi).$ Then to prove the second inequality $k(\Lambda)\leq \cd(\phi),$ it suffices to show $k(\Lambda)\leq\cd(\psi).$ 

We do it by induction on $k=k(\Lambda).$ 
When k=1 we have $\Lambda=\Z_{n_1}$. Then the homomorphism $\psi^{*}=\psi_1^{*}:H^{1}(B\Z_{n_{1}};Z_{n_{1}})\to H^{1}(B\Z;Z_{n_{1}})$ is nonzero, since $\psi:\Z\to\Z_{n_{1}}$ is surjective.  

Suppose the result holds true for all $l\leq k$. 
First we note that by Theorem \ref{IDF}
the group $\Lambda$ for $k(\Lambda)=k+1$ is written uniquely as $\Lambda=\Z_{n_{1}}\times\cdots\Z_{n_{k+1}}$ with $n_{1}|\cdots|n_{k+1}.$
Note also that $B\Lambda$ can be presented as the product
$B\Z_{n_1}\times\dots\times B\Z_{n_{k+1}}$.
Let $p$ be a prime that divides $n_1$ and, hence, all $n_i$.
We show that the induced homomorphism $$\psi^{*}:H^{k+1}(B\Z_{n_{1}}\times ...\times B\Z_{n_{k+1}};\Z_p )\to H^{k+1}(T^{k+1};\Z_p)$$ is a nonzero homomorphism.

It is known that the integral cohomology groups $H^{j}(B\Z_m;\Z)$ are $\Z_m$ if $j$ is even and zero otherwise~\cite{Ha}. Note that for prime $p$ dividing $m$
by the Universal Coefficient Formula $H^{j}(B\Z_m;\Z_{p})=\Z_{p}$ for all $j$, since $\Z_m\otimes\Z_{p}=\Z_{p}$ and $\text{Tor}(\Z_m,\Z_{p})=\Z_{p}.$ Thus, for prime $p$ dividing $n_1$
we obtain $H^{j}(B\Z_{n_{i}};\Z_{p})=\Z_{p}$ for all $i$ and $j$. 
Since for each $n_{i}$ the CW-complex $B\Z_{n_{i}}$ are of finite type, we can apply the Kunneth Formula~\ref{KFT}.
By the Kunneth Formula with a field coefficient $\Z_{p},$ and induction, we get that $H^{j}(B(\Z_{n_{1}}\times ...\times \Z_{n_{k+1}});\Z_{p})$ is not zero for all $j$.  Clearly for the $(k+1)$-torus $T^{k+1}$ we have $H^{k+1}(T^{k+1};\Z_{p})=\Z_{p}.$  Using commutative diagram below, we get that $\psi$ is a nonzero homomorphism for the mod $p$ cohomology in dimension $k+1$.
$$
\tikzset{node distance=0.005mm, auto}
\begin{tikzcd}[scale=1]
H^{k}(B\Z_{n_{1}}\times ...\times B\Z_{n_{k}});\Z_{p})\otimes H^{1}(B\Z_{n_{k+1}};\Z_{p}) \arrow[r, "\times"] \arrow[d, "\psi^{*}\otimes \psi^{*}"] &  H^{k+1}(B\Z_{n_{1}}\times ...\times B\Z_{n_{k+1}});\Z_{p}) \arrow[d, "\psi^{*}"]\\
H^{k}(T^{k};\Z_{p})\otimes H^{1}(S^{1};\Z_{p}) \arrow[r, "\times"]& H^{k+1}(T^{k+1};\Z_{p}) \\
\end{tikzcd}
$$
Indeed,
the horizontal maps are isomorphism, by the Kunneth Theorem. Thus the Kunneth Formula isomorphism takes the tensor product to the cross product, $a\otimes b \overset{\times}{\rightarrow} a\times b.$
Here the cross product is defined as $a\times b=p^{*}_{1}(a)\cup p^{*}_{2}(b)$ where $p_{1}$ and $p_{2}$ are the projections of the product $X\times Y$ onto $X$and $Y$ respectively. Using the naturality of the cup product and the induction assumption, we obtain:
$$\psi^{*}(a\times b)=\psi^{*}(p^{*}_{1}(a)\cup p^{*}_{2}(b))=(\psi^{*}\degree p^{*}_{1})(a)\cup(\psi^{*}\degree p^{*}_{2})(b)=\psi^{*}(a)\otimes\psi^{*}(b)\neq 0$$
Hence, $\cd(\psi)=k+1.$ 
\end{proof}

We use the notation $T(A)$ for the torsion subgroup of an abelian group $A$.
For finitely generated abelian groups we define the rank $rank(A)=rank(A/T(A))$.

\begin{lemm}\label{split}
Every epimorphism $\phi:\Z^n\to\Z^m\oplus \Lambda$ splits as the direct sum $$\phi=\psi_1\oplus\psi_2:\Z^m\oplus\Z^{n-m}\to\Z^m\oplus \Lambda.$$
\end{lemm}
\begin{proof}
The epimorphism $\phi$ as a map to the product is defined by coordinate functions,
$\phi=(\phi_1,\phi_2)$ which are also epimorphisms.
There exists a section $s:\Z^{m}\to\Z^{n}$ of the epimorphism $\phi_1$, since $\Z^{m}$ is free abelian. 
We show that $\Z^{n}$ splits as the direct sum $s(\Z^{m})\oplus\phi^{-1}(\Lambda).$ 
For each element $x\in \Z^{n}$ we have $\phi_1(x-s\phi_1(x))=\phi_1(x)-\phi_1s(\phi_1(x))=0.$ Therefore,
$x-s\phi_1(x)\in \phi^{-1}(\Lambda)$. Thus, every element $x\in \Z^{n}$
can be written as $s\phi_1(x)+(x-s\phi_1(x))$.
Suppose that $y\in s(\Z^{m})\cap\phi^{-1}(\Lambda).$ Since  $\phi(y)\in G$,
we obtain $\phi_1(y)=0$. Since $y\in s(\Z^m)$, we have $s\phi_1(y)=y$. Hence $y=0$
and the sum $s(\Z^m)+\phi^{-1}(\Lambda)=\Z^n$ is a direct sum.

Note that $s(\Z^m)\cong \Z^m$. In view of the splitting $\Z^n=\Z^m\oplus \phi^{-1}(\Lambda)$ it follows that $rank(\phi^{-1}(\Lambda))=n-m$.
 We define $\psi_{1}$ and $\psi_{2}$ to be the restrictions of $\phi_1$ and $\phi_2$
 to corresponding direct summands. i.e.,  $\psi_{1}:=\phi_1|_{s(\Z^{m})}$ and $\psi_{2}:=\phi_2|_{\phi^{-1}(\Lambda)}.$  
\end{proof}

\begin{thm}{\label{Abelian}}
Let $\phi:\Gamma\to\Lambda$ be an epimorphism of finitely generated abelian groups with $\phi(T(\Gamma))=0$. Then $$\cat(\phi)=\cd(\phi).$$ 

In particular, 
$$\cat(\phi)=\cd(\phi)=rank(\Lambda)+k(T(\Lambda)),$$ where $k(T(\Lambda))$ is the Smith normal number of $T(\Lambda)$.
\end{thm}

\begin{proof}
Since the groups $\Gamma$ and $\Lambda$ are finitely generated abelian groups, we may assume $\Gamma=\Z^{n}\oplus T(\Gamma)$ and $\Lambda=\Z^{m}\oplus T(\Lambda)$ for some $m$ and $n$. 

If both groups $\Gamma,\Lambda$ have torsion subgroups
and $\phi(T(\Gamma))\ne 0$, then Question \ref{question} does not hold as Dranishnikov's example \ref{Example} shows.

Thus, we consider $\phi$ with $\phi(T(\Gamma))=0$.
Such $\phi$ factors through the epimorphism $\bar\phi:\Z^n\to\Lambda$. 
In view of the retraction $B\Gamma\to B\Z^n$ we obtain that
$\cd(\bar\phi)=\cd(\phi)$ and $\cat(\bar\phi)=\cat(\phi)$. Thus, we may assume that
$\Gamma=\Z^n$.

If $\Lambda$ has no torsion, then  $\cat(\phi)=\cd(\phi)=rank(\phi)=rank(\Lambda)=m$
by Jamie Scott's result ~\cite{Sc}.
We consider the case $T(\Lambda)\ne 0$.   
Let $\phi:\Z^{n}\to\Z^{m}\oplus T(\Lambda)$ be an epimorphism. 
By Lemma~\ref{split} $\phi$ breaks into the direct sum $\psi_{1}\oplus\psi_{2}$ where $\psi_{1}:\Z^{m}\to\Z^m$ is an isomorphism and an epimorphism $\psi_{2}: \Z^{n-m}\to T(\Lambda).$  

By the well-known inequality for LS category of product of maps in ~\cite{CLOT}, we obtain $$\cat(\psi_{1}\oplus\psi_{2})\leq \cat(\psi_{1})+\cat(\psi_{2})=m+k,$$ where $k=k(T(\Lambda))$, since by Lemma ~\ref{Smith}
$\cat(\psi_{2})=\cd(\psi_{2})=k(T(\Lambda))$. 

To complete the proof, we show that $m+k\le \cd(\phi)$. By Theorem ~\ref{IDF}, the torsion group $T({\Lambda})$ admits a decomposition $T(\Lambda)=\Z_{n_{1}}\times..\times \Z_{n_{k}}$ where $n_{1},\cdots,n_{k}$ are natural numbers with $n_{1}|\cdots|n_{k}.$  The proof of Lemma \ref{Smith}
gives a nonzero homomorphism $$\psi_2^*:H^k(BT(\Lambda);\Z_p)\to H^k(T^{n-m};\Z_p).$$ We apply the Kunneth Formula (Theorem ~\ref{KFT}) with $\Z_{p}$ coefficients for $p|n_{1}$ to obtain a nonzero homomorphism  
$$\phi^{*}:H^{m+k}( T^m\times BT(\Lambda);\Z_p)\to H^{m+k}(T^m\times T^{n-m};\Z_p).$$ 
Therefore, $\cd(\phi)\ge m+k$.
This proves the theorem.
\end{proof}

\section*{Acknowledgments}
I would like to thank my advisor, Alexander Dranishnikov, for all of his help and encouragement throughout this project. I am grateful to two anonymous referees in Algebra and Discrete Mathematics for invaluable comments that greatly improved the quality of the manuscript.


\footnotesize
\begin{thebibliography}{999999}
\bibliographystyle{alpha}


\bibitem[AS]{AS} M.F. Atiyah and G. B. Segal, {\em Equivariant K-theory and completion}, J. Differential Geometry \textbf{3} (1969), 1-18.


\bibitem[Br]{Br} K. Brown, Cohomology of Groups. \emph{Graduate Texts in Mathematics},
\textbf{87} Springer, New York Heidelberg Berlin, 1994.


\bibitem[CLOT]{CLOT}
    O. Cornea, G. Lupton, J. Oprea, D. Tanre,
\newblock   { Lusternik-Schnirelmann Category},  AMS,  2003.

\bibitem[D]{D} A. N. Dranishnikov, {\em On some problems related to the Hilbert-Smith conjecture}, Mat. Sb. \textbf{207} (2016), no. 11, 82-104.

\bibitem[DD]{DD} Aditya De Saha, Alexander Dranishnikov, On cohomological dimension of group homomorphisms
preprint, arXiv:2302.09686 [math.AT] (2023)

\bibitem[DF]{DF} 
D. Dummit, R.Foote, \emph{Abstract algebra}, Wiley Hoboken, 2004.

\bibitem[Do]{Do} A. Dold, Lectures on algebraic topology. Springer Science \& Business Media; 2012.

\bibitem[DK]{DK} Dranishnikov, Alexander, Nursultan Kuanyshov. ``On the LS-category of group homomorphisms." Mathematische Zeitschrift 305, no. 1 (2023): 14.


\bibitem[EG]{EG} S. Eilenberg, T. Ganea, {\em On the Lusternik-Schnirelmann Category of Abstract Groups.} Annals of
Mathematics, 65, (1957), 517-518.

\bibitem[Gr]{Gr}
    M. Grant, 
\newblock    { https://mathoverflow.net/questions/89178/cohomological-dimension-of-a-homomorphism}

\bibitem[Ha]{Ha} A. Hatcher, Algebraic topology, 2005.

\bibitem[Hu]{Hu} T. Hungerford, \emph{Algebra (Vol. 73)}, Springer, 2012.

\bibitem[Ku1]{Ku1} N. Kuanyshov, On the LS-category of homomorphism of almost nilpotent groups, Topology and its Applications, Volume 342, 1 February 2024, 108776.

\bibitem[Ku2]{Ku2} N. Kuanyshov, On the sequential topological complexity of group homomorphisms. arXiv preprint arXiv:2402.13389 (2024).

\bibitem[LS]{LS}  L. Lusternik, L. Schnirelmann, ``Sur le probleme de trois geodesiques fermees sur les surfaces de genre 0", Comptes Rendus de l'Academie des Sciences de Paris, 189: (1929) 269-271.
 
\bibitem[Sc]{Sc} J. Scott. {\em On the topological complexity of maps}, Topology and its Applications 314 (2022), Paper No. 108094, 25 pp.

 \bibitem[Sp]{Sp} E. Spanier, Algebraic topology, Springer Science \& Business Media 1989.



\end{thebibliography}
\end{document}